\documentclass[a4paper]{article}
\usepackage{amsthm,amssymb,amsmath,enumerate,mathtools,tikz,tikz-3dplot,xcolor} 
\usepackage{graphicx,subcaption}
\usetikzlibrary{arrows}
\usetikzlibrary{arrows,decorations.pathmorphing,backgrounds,positioning,fit,matrix}
\usetikzlibrary{decorations.pathmorphing}
\usetikzlibrary{decorations.markings}
\tikzset{snake it/.style={decorate, decoration=snake}}
\usetikzlibrary{arrows,automata,positioning}
\usepackage{pst-node}
\usepackage{tikz-cd}
\usepackage{float}

\tikzset{
	>=stealth',
	punkt/.style={
		rectangle,
		rounded corners,
		draw=black, very thick,
		text width=6.5em,
		minimum height=2em,
		text centered},
	pil/.style={
		->,
		thick,
		shorten <=2pt,
		shorten >=2pt,}
}
\usepackage{hyperref}
\usepackage{comment,marginnote}
\usepackage{cleveref}

\newcommand{\N}{\ensuremath{\mathbb{N}}}

\newcommand{\Z}{\ensuremath{\mathbb{Z}}}

\newcommand{\cay}{\ensuremath{\mathsf{Cay}}}
\newcommand{\dih}{\ensuremath{\mathsf{Dih}}}
\newcommand{\sm}{\smallsetminus}

\newcommand{\Dinf}{\ensuremath{D_\infty}}

    \newcommand{\G}{\Gamma}

\newcommand{\free}{\ast\!\! }

\newcommand{\W}{\overline{W}}
\newcommand{\La}{\overline{\Lambda}}
\newcommand{\T}{\overline{\mathcal{T}}}

\newtheorem{theorem}{Theorem}[section]
\newtheorem{lemma}[theorem]{Lemma}

\newtheorem{thm}[theorem]{Theorem}
\newtheorem{coro}[theorem]{Corollary}
\newtheorem{remark}[theorem]{Remark}

\newtheorem{definition}[theorem]{Definition}

\newtheorem{prob}{Problem}

\newtheorem{conj}{Conjecture}
\newtheorem*{conj*}{Conjecture}
\newtheorem*{exa}{Example}

\theoremstyle{remark}
\newtheorem{claim}{Claim}

\newcounter{case}

\renewcommand{\thecase}{\Roman{case}}

\newenvironment{txteq*}
{
	\begin{equation*}
		\begin{minipage}[t]{0.85\textwidth} 
			\em                                
		}
		{\end{minipage}\end{equation*}\ignorespacesafterend}

\title{Hamiltonicity in generalized quasi-dihedral groups
}
\author{Babak Miraftab\thanks{Corresponding author} \\ \href{mailto:bobby.miraftab@uleth.ca}{babakmiraftab@cunet.carleton.ca} \and Konstantinos Stavropoulos\\ \href{mailto:konstantinos.stavrop@gmail.com}{konstantinos.stavrop@gmail.com}}
	\date{\today}

\begin{document}
 \maketitle

\begin{abstract}
Witte Morris showed in [Discrete Math, 38.1 (1982)] that every connected Cayley graph of a finite (generalized) dihedral group has a Hamiltonian path. The infinite dihedral group is defined as the free product  $\Z_2 \ast \Z_2$. 
We show that every connected Cayley graph of the free product with amalgamation $A \ast_K A$ has  a Hamiltonian double ray, where $A$ is 
a generalized quasi-dihedral group on $K$.
Additionally, this leads to the conclusion that each connected Cayley graph of the infinite dihedral group also contains a Hamiltonian double ray.
\end{abstract}

\section{Introduction}
A \textit{Hamiltonian cycle(path)} in a finite graph is a cycle(path) which includes every vertex of the graph.
A graph $G$ is  \textit{vertex-transitive} if for any two vertices $v_1$ and $v_2$ of $G$, there is some automorphism $f\colon G\to G$
such that $f(v_{1})=v_{2}$. 
The Lov\'{a}sz conjecture for vertex-transitive graphs states that every finite connected vertex-transitive graph contains a Hamiltonian cycle except five known counterexamples. 
Nevertheless, even the weaker version of the conjecture for finite Cayley graphs remains unresolved. For a comprehensive survey on the field, see~\cite{Lanel,wittesurvey}.

A \textit{dihedral group} is the group of symmetries of a regular polygon, which includes rotations and reflections.
Dihedral groups contain a cyclic subgroup with  an index of $2$.
A group is called \emph{generalized quasi- dihedral} if there is an abelian subgroup $K$ of index $2$ such that there is an element $b \in G \sm  K$ such that $bkb^{-1}k=1$ for every $ k \in K$ and $b^2\in K$.
Witte showed in \cite{wittedigraphs} that every connected Cayley graph of a finite (generalized) dihedral group has a Hamiltonian path. 
It is worth mentioning that the existence of a Hamiltonian cycle in a dihedral group is still not known.

While all preceding results concerned finite graphs, Hamiltonian cycles(paths) have also been considered in infinite graphs. While a Hamiltonian ray may be seen as a generalization of a Hamiltonian path for infinite graphs, determining the correct infinite analogue of a Hamiltonian cycle remains controversial.
One such candidate is a \emph{double-ray}, an infinite 2-regular connected graph. A double-ray is called hamiltonian for a graph $G$, if it is spanning inside $G$.

It is an easy observation to see that the Lov\'{a}sz conjecture already fails for infinite Cayley graphs with Hamiltonian double-rays in place of Hamiltonian cycles, since the amalgamation of more than $k+1$ groups on a subgroup of order $k$ will produce groups with Cayley graphs that have separators of size $k$ whose removal leaves more than $k+1$ components, a well-known obstruction to Hamiltonicity (see~\cite{AgelosFleisch}).

However, the above counterexamples have infinitely many ends.
On the other hand, it was proven in~\cite{MR4908254,HC1,nash1959abelian} that any Cayley graph of an abelian two-ended group contains a Hamiltonian double ray. In particular, the following has been conjectured.

\begin{conj}[\cite{miraftab2017cycles}]\label{conjecture}
Any Cayley graph of a group with at most two ends has a Hamiltonian double ray.
\end{conj}

In this paper, we address the result of Witte \cite{wittedigraphs} for infinite dihedral groups as well as a natural generalization of them: two-ended generalized quasi-dihedral groups --- the definition of which we defer to the next Section --- and thus, we make progress towards~\Cref{conjecture}.
Our main result is the following:

\begin{thm}\label{main}
Let $G=A\ast_K A$, where  $A$ is a  generalized quasi-dihedral on $K$.
Then every connected Cayley graph of $G$ contains a Hamiltonian double ray. 
\end{thm}

We note that $[A:K]=2$ which implies that $G$ is two-ended.
\noindent
The following Corollary is an immediate consequence of~\Cref{main}.

\begin{coro}
Every connected Cayley graph of the infinite dihedral group $\Z_2\ast \Z_2$ has a Hamiltonian double ray.
\end{coro}

Our proof relies on extracting infinite two-ended grids or infinite two-ended cubic ``cylindrical walls'' (the precise definition of which we defer to~\Cref{sec:walls}) as spanning subgraphs of the Cayley graph. We prove that they, in turn, contain Hamiltonian double rays.

\section{Preliminaries}
In this paper, all groups are assumed to be finitely generated.
We begin with the definition of generalized dihedral groups and  generalized dicyclic group.
The \emph{generalized dihedral group on an abelian subgroup $K=\langle S\mid R\rangle$} is the group $\dih(K):=\langle S,b\mid R,b^2=bkb^{-1}k=1,\forall k\in K\rangle$. 
Alternatively, it is the external semidirect product $\dih(K) = K \rtimes_{\phi} \Z_2$, where $K$ is abelian and $\phi(1)(g) =-g$ for any $g \in K$, in additive notation.
When $K=\Z$, we obtain the infinite dihedral group $D_{\infty}$. 
Note that an alternative presentation of $D_{\infty}$ is $\langle a,b\mid a^2=b^2=1\rangle=\Z_2\ast \Z_2$, where $\ast$ denotes the free product of groups.
Let $K$ be an abelian group of even order and of exponent greater than $2$, and let $y$ be an involution of $K$.
The \emph{generalized dicyclic group} is the group $\langle K, b \mid b^2 = y, b^{-1}kb = k^{-1}, \forall k \in  K\rangle $, see \cite{MR3266284}.
Next, we extend these definitions to a more general setting.
\begin{definition}\label{semidihedral}
A group $G$ is \emph{generalized quasi-dihedral on a subgroup $K$  and $b\in G\sm K$} if the following holds:
\begin{itemize}
    \item $K$ is an abelian subgroup and $[G:K]=2$, and 
    \item There is an element $b \in G \sm  K$ such that $bkb^{-1}k=1$ for $\forall k \in K$.
\end{itemize}  
For the sake of simplicity, we use the abbreviation \emph{GQD} on $K$ and $b$ to refer to a generalized quasi-dihedral group on $K$ and $b$.
\end{definition}

\begin{exa}
A group G is called a generalized dicyclic group, written as Dic(A, y), if it is generated by A and an additional element x, and in addition we have that $[G:A] = 2$, $x^2 = y$, and for all $a$ in $A$, $x^{-1}ax = a^{-1}$. 
For any abelian group $H$, the generalized dihedral group of $H$, written Dih($H$), is the semidirect product of $H$ and $\Z_2$, with $\Z_2$ acting on H by inverting elements. 
$$Dih(H)=C_{2}\ltimes H=\langle s,H\mid shs^{-1}=h^{-1}\text{\ for\ all\ }h\in H,s^{2}=1\rangle$$
\end{exa}
\begin{lemma}
    If $G$ is a generalized quasi-dihedral group on a subgroup $K$ and a specific element $b\in G\sm K$, then $G$ is also generalized quasi-dihedral on $K$ and $x$ for every $x\in G\sm K$.
\end{lemma}

\begin{proof}
Indeed, let $x=kb\in Kb$. 
Then $xk'x^{-1}=kbk'b^{-1}k^{-1}=kk'^{-1}k^{-1}=k'^{-1}$.
\end{proof}

\begin{remark}
    Thus we will say that $G$ is \emph{generalized quasi-dihedral} just on $K$ when we do not want to fix $b$ outside of $K$.
\end{remark}


\begin{definition}
Let $G_1=\langle S_1\mid R_1\rangle$ and $G_2=\langle S_2\mid R_2\rangle$ be two groups.
Suppose that a subgroup $H_1$ of $G_1$ is isomorphic to a subgroup $H_2$ of $G_2$, say an isomorphic map $\phi\colon H_1\to H_2$.
The  \emph{free product with amalgamation} of $G_1$ and $G_2$ over $H_1\cong H_2$ is
 $$G_1 \underset{H_1}{\ast} G_2=\langle S_1\cup S_2\mid R_1, R_2, h\phi(h)^{-1}=1,\forall h\in H_1\rangle. $$
\end{definition}
\noindent For more details and applications of free products with amalgamation, see \cite{miraftab2019splitting}.

\noindent A graph $\Gamma$ is called \emph{locally finite} if every vertex has finite degree. 
\begin{definition}[cf.\ {\cite[pp.~1424]{RDsBanffSurveyI}}]
We call $1$-way infinite paths \emph{rays}, and $2$-way infinite paths \emph{double rays}. 
An \emph{end} of $\Gamma$ is an equivalence class of rays in $\Gamma$,
where two rays are considered \emph{equivalent} if no finite set of vertices separates them in $\Gamma$.
\end{definition}

Let $G$ be a finitely generated group. 
Let $S \subseteq G$ be a finite generating set of $G$ and let $\cay(G;S)$ be the Cayley graph of $G$ with respect to $S$. 
The number of ends of $G$ is defined as the number of ends of $\cay(G;S)$. 
A basic fact in the theory of ends for groups states that the number of ends of $G$ does not depend on the choice of a finite generating set $S$ of $G$, so that it is well-defined, see Corollary 2.3 of \cite{miraftab2018two}.

\begin{thm}{\rm\cite[Theorem 5.12]{ScottWall}}\label{classifiication}
Let $G$ be a finitely generated group. Then the following statements are equivalent:
\begin{enumerate}[\rm (i)]
    \item $G$ is a two-ended group.
    \item $G$ is isomorphic to either the free product with amalgamation $A\ast_C B$, where $C$ is finite and
$[A : C] = [B : C] = 2$ or the HNN-extension $\ast_{\phi}C$, where $C$ is finite and $\phi \in Aut(C)$.
\end{enumerate}
\end{thm}	

Lastly, we reformulate~\Cref{classifiication} in the case that a two-ended group is a free product with amalgamation.

\begin{coro}\label{corclass}
Let $G$ be a finitely generated group. The following are equivalent:
\begin{enumerate}[\rm (i)]
    \item $G$ is two-ended and $G=A\ast_C B$, where $A$ and $B$ are finite.
    \item $G=A\ast_C B$, where $C$ is finite and $[A:C]=[B:C]=2$.
    \item $G/C\cong D_{\infty}$, where $C$ is finite and normal in $G$.
\end{enumerate}
\end{coro}

\begin{proof}
(i) $\Leftrightarrow$ (ii): The inverse implication follows directly from  Lemma~\ref{classifiication}. For the forward implication, let $\Gamma$ be the Cayley graph of $G$ with respect to the generating set $A\cup B$. Assume that $[A:C]\geq 3$ or $[B:C]\geq 3$. Recall that $C \leq A, B$, hence also finite. Then $\Gamma\sm A$ has at least three infinite components, each of which must contain an end by the fact that $\Gamma$ is a Cayley graph, a contradiction to $G$ being two-ended.\\
(ii)  $\Leftrightarrow$ (iii): We easily deduce by the presentation of a free product with amalgamation that if $C$ is normal in $G$, we have that $G=A\ast_C B$ if and only if $G/C\cong (A/C) \ast (B/C)$. We deduce (ii) from (iii) by recalling that $D_{\infty}\cong \Z_2\ast \Z_2$. Finally, (ii) implies (iii) by noting that having index $2$, $C$ is normal in both $A$ and $B$. Since $A$ and $B$ generate $G$ (by definition), we infer that $C$ is normal in $G$.
\end{proof}

\section{Properties of GQD groups}
Again we assume that all groups here are finitely generated. 
The following lemma follows directly by the definition of a generalized quasi-dihedral group.

\begin{lemma}\label{b4=1}
Let $G$ be a generalized quasi-dihedral on $K$ and $b$. 
Then the following statements hold:
\begin{enumerate}[\rm(i)]
    \item If $x\in G\sm K$, then $x^4=1$.
    \item Every subgroup of $K$ is normal in $G$.
\end{enumerate}
\end{lemma}

\begin{proof}
\begin{enumerate}[\rm(i)]
    \item Since $[G:K]=2$, one can see that $x^2\in K$.
    By~\Cref{semidihedral}, we have $xx^2x^{-1}=x^{-2}$, or equivalently $x^4=1$.
    \item This part is a direct consequence of the definition.\qedhere
\end{enumerate}
\end{proof}

In \cite{wittedigraphs}, Witte Morris proved the following:

\begin{thm}{\rm\cite[Theorem 5.1]{wittedigraphs}}\label{finitehamiltonpath}
If a finite group $G$ has a subgroup $N$ of index 2 such that every subgroup of $N$ is normal in $G$, then every Cayley graph of $G$ has a Hamiltonian path.
\end{thm}

As a combination of~\Cref{b4=1} and~\Cref{finitehamiltonpath} we obtain the following corollary.

\begin{coro}\label{semiham}
Let $G$ be a finite generalized quasi-dihedral group.
Then every Cayley graph of $G$ has a Hamiltonian path.\qed
\end{coro}

Our ultimate goal is to extend the above result to infinite generalized quasi-dihedral groups.
We present some of their general properties, which we need towards the main result of the paper.

\begin{lemma}\label{b2}
    Assume that $A$ is a generalized quasi-dihedral on $K$ and $b$ and $B$ is a generalized quasi-dihedral on $K$ and $b'$.
    Then there is an isomorphism $\varphi\colon A\to B$, where $\varphi$ is the identity map on $K$ if and only if $b^2=b'^2$.
\end{lemma}

\begin{proof}
If $\varphi$ exists and $\varphi|_K=\mathrm{id}_K$, then
\[
b^2=\varphi(b^2)=\varphi(b)^2=(b')^2,
\]
so $b^2=(b')^2$.
Conversely, we assume that $b^2=b'^2$.
Define $\varphi\colon A\to B$ by sending $kb^i$ to $kb'^i$, where $i=0,1$.
We show that $\varphi$ is an isomorphism. It is straightforward to show that  $\varphi(xy)=\varphi(x)\varphi(y)$ for every $x,y\in A$.
We show that $\varphi(x^{-1})=\varphi(x)^{-1}$ for $x\in A$.
If $x\in K$, then we are done. 
Else,  $x=kb$ and
\begin{align*}
    \varphi((kb)^{-1})&=\varphi(b^{-1}k^{-1})\\
    &=\varphi(kb^{-1})\\
    &=\varphi(kb^3)=kb^2b'\\
    &=kb'^2b'=kb'^{-1}\\
    &=b'^{-1}k^{-1}\\
    &=\varphi(kb)^{-1}
\end{align*}
\end{proof}

\begin{lemma}\label{bb'}
Let $G=A\ast_K B$, where $A$ is GQD on $K$ and $b$ and $B$ is GQD on $K$ and $b'$.
If $k\in K$ and $n\in \Z$, then:
\begin{enumerate}[\rm(i)]
\item $(b'b)^n k = k(b'b)^n$ and $(b'b)^n = k'(bb')^{-n}$ for some $k'\in K$;
\item $(bb')^n k = k(bb')^n$ and $(bb')^n = k''(b'b)^{-n}$ for some $k''\in K$.
\end{enumerate}
\end{lemma}

\begin{proof}
First, since $b'kb'^{-1}=k^{-1}$ and $bkb^{-1}=k^{-1}$ for all $k\in K$, we have
\[
b'b\,k = b'(bk) = b'(k^{-1}b)= (b'k^{-1})b = k\,b'b,
\]
and similarly $bb'\,k = k\,bb'$. Hence $(b'b)^n$ and $(bb')^n$ commute with every $k\in K$.

Next, note that $b^{-1}=b^3=b^2b$ and $b'^{-1}=b'^3=b'^2b'$ (since $b^4=b'^4=1$ and $b^2,b'^2\in K$).
Thus
\[
(bb')^{-1}=b'^{-1}b^{-1}=(b'^2b')(b^2b)=(b'^2b^2)\,(b'b),
\]
so with $h\coloneqq b'^2b^2\in K$ we get $(bb')^{-1}=h\,(b'b)$ and hence
\[
bb'=(b'b)^{-1}h^{-1}.
\]
Because $h\in K$ commutes with $b'b$, we obtain for all $n\in\Z$:
\[
(bb')^n = (b'b)^{-n}\,h^{-n}.
\]
This gives (ii) with $k''=h^{-n}\in K$. The proof of (i) is analogous (swap $b$ and $b'$).
\end{proof}

\begin{thm} {\rm \cite[Theorem 11.3]{bogo}}\label{normalform}
Let $G_1$ and $G_2$  be two groups with isomorphic subgroups~$H_1$ and $H_2$ respectively. 
Let~$T_{i}$ be a left transversal\footnote{A \emph{transversal} is a system of representatives of left cosets of~$H_i$ in~$G_i$ and we always assume that~$1$ belongs to it.} of~$H_i$ for~${i=1,2}$.
Any element $x\in G_1\free_{H} G_2$ can be uniquely written in the form  $x=x_0x_1\cdots x_n$, where $H\cong H_1\cong H_2$:
\begin{itemize}
\item[{\rm(i)}] $x_0\in H_1$.
\item[{\rm(ii)}]$x_j\in T_1\sm 1$ or $x_j\in T_2 \sm 1$ for $j\geq 1$ and the consecutive terms $x_j$ and $x_{j+1}$ lie in distinct transversals.
\end{itemize}
\end{thm}

We denote the negative integers by $-\mathbb{N} $.

\begin{lemma}\label{randomelement}
Let $G=A\ast_K B$, where $A$ is generalized quasi-dihedral on $K$ and $b$ and $B$ is generalized quasi-dihedral on $K$ and $b'$.
If $g\in G$ is an arbitrary element of $G$, then $g$ has one of the following forms, where $k\in K$:
\begin{itemize}
\item[{\rm(i)}] $g=k(bb')^n$ for $n\in \Z$.
\item[{\rm(ii)}] $g=k(bb')^nb$ for $n\in \N\cup \{0\}$.
\item[{\rm(iii)}] $g=k(bb')^nb'$ for $n\in -\N\cup \{0\}$.
\end{itemize}
\end{lemma}

\begin{proof}
By \Cref{normalform}, every element $g\in G$ can be expressed in one of the following forms:
\[
k(bb')^n,k(bb')^nb,k(b'b)^n,k(b'b)^nb',
\]
for some $n\in \N$ and $k\in K$.
If $g$ is either $k(b'b)^n$ or $g=k(b'b)^nb'$, then by \Cref{bb'} $g$ can be written as either  $k'(bb')^{-n}$ or $k'(bb')^{-n}b'$.
\end{proof}

\begin{lemma}{\rm\cite[Lemma 5.6]{ScottWall}}\label{finiteindex}
Let $H$ be a subgroup of finite index in $G$.
Then the number of ends of $H$ is the same as the number of ends of $G$.
\end{lemma}

\begin{theorem}\label{equivalence}
Let $G=A\ast_K A$, where  $A$ is a generalized quasi-dihedral on $K$.
Then $G$ is generalized quasi-dihedral.
Conversely, if $G$ is two-ended generalized quasi-dihedral, then $G=A\ast_K B$, where $A$ and $B$ are generalized quasi-dihedral on $K$.
\end{theorem}

\begin{proof}
For the forward implication, 
let $G=A\ast_K B$, where $A\cong B$ and so $b^2=b'^2$.
Let $a=bb'$. 
Recall that by~\Cref{corclass}, $K$ is a normal subgroup of $G$.
By \Cref{bb'}, we have $a^nk=(bb')^nk=k(bb')^n=ka^n$, so $K\langle a\rangle$ is an abelian subgroup of $G$.

\begin{claim}\label{[G:Ka]=2}
 $[G:K\langle a\rangle]=2$.
\end{claim}
\begin{proof}[Proof of the Claim]
It follows from Lemma 3.6 that $G = K\langle a\rangle \cup K\langle a\rangle b \cup K\langle a\rangle b'$. 
Hence to prove the claim it is enough to show that $b'$ belongs to $K\langle a\rangle b$ (and that $K\langle a\rangle \neq G)$. 
But this is simply the fact that $b'b^{-1}a = b'b^{-1}(bb') = b'^2$ is in $K$.
\end{proof}
Since $A$ is generalized quasi-dihedral on $K$ and $b$, we have $bkb^{-1}k=1$ for every $k\in K$.
In fact, we show that $bxb^{-1}x=1$ for every $x\in K\langle a\rangle$.
This will follow by the following claim:

\begin{claim}\label{bak=(ka)^{-1}b}
$b(kbb')=(kbb')^{-1}b$.
\end{claim}
\begin{proof}[Proof of the Claim]
As before, we have
\begin{align}
    bkbb' &=k^{-1}bbb'\nonumber\\
          &=k^{-1}b'b^{-2}\label{ww}\\
          &=k^{-1}b'^{-3}b^{-2}\label{www}\\
          &=k^{-1}b'^{-1}b'^{-2}b^{-2}\nonumber\\ 
          &=b'^{-1}k\label{wwww}\\ 
          &=b'^{-1}(b^{-1}b)k\nonumber\\ 
          &=b'^{-1}b^{-1}k^{-1}b\nonumber\\ 
          &=(kbb')^{-1}b\nonumber
\end{align} 

\Cref{ww} follows from the fact that $b^2\in K$ and so $b^2b'=b'b^{-2}$, whereas \Cref{www} follows from~\Cref{b4=1}(i).
Lastly, for \Cref{wwww} we simply make use of $b^2=b'^2$.
\end{proof}

\noindent We have shown that $G$ is a generalized quasi-dihedral group on $K\langle a\rangle$ and $b$.\\

Conversely, assume that $G$ is a two-ended group and moreover $G$ is a generalized quasi-dihedral group on $H$ and $b$ and so $[G:H]=2$ and $b^4=1$.
It follows from~\Cref{finiteindex} that $H$ is a two-ended abelian group.
So $H$ must be isomorphic to $K\langle a\rangle$, where $K$ is the finite maximal subgroup of $H$ and the order of $a$ is infinite.
It follows from \Cref{b4=1}(ii) that $K$ and $\langle a\rangle$ are normal subgroups of $G$.
We claim the following:
\begin{claim}
$G/K\cong D_{\infty}$.
\end{claim}
\begin{proof}[Proof of the Claim]
Let $\bar{G}\coloneqq G/K$ and $\bar{H}\coloneqq H/K$.
Hence $[\bar{G}: \bar{H}]=2$ and $\bar{H}\cong \Z$, which implies that $\bar{G}$ is either $\Z\times \Z_2$ or $\Z\rtimes \Z_2$.
We exclude the former case by showing that $\bar{G}$ is not abelian.

Indeed, assume that $xK$ and $yK$ commute for every $x,y \in G$.
Since $G$ is a generalized quasi-dihedral group on $H$, we have $xhx^{-1}h=1$ for every $x\in G\sm H$ and $h\in H$. 
In particular, $a\in H$ and so $xax^{-1}a=1$ for every $x\in G\sm H$.
However, $xax^{-1}a^{-1}K=K$, as $aK$ and $x^{-1}K$ commute.
We obtain 
\begin{align*}
xax^{-1}a^{-1}K=&1K\\  
xax^{-1}a^{-1}K=&xax^{-1}aK\\    
a^{-1}K=&aK, 
\end{align*}
which implies that $a^2\in K$, violating the fact that $K$ is finite and $a$ has an infinite order.
Thus $\bar{G}$ is not abelian and so $\bar{G}\cong \Z\rtimes \Z_2 \cong D_{\infty}$.
\end{proof}
It follows from~\Cref{corclass} that the group $G$ can be written as $A\ast_{K} B$, where $[A:K]=[B:K]=2$.
Lastly, we show that $A=cK\sqcup K$ is generalized quasi-dihedral on $K$, where $c\in A\sm K$.
Indeed, any $c=bh \in bH\cap A$ will do, as $bhb^{-1}h=1$ and $G=H\sqcup bH$. 
With an analogous method, one can show that $B$ is also generalized quasi-dihedral.
\end{proof}

We gather some useful facts in the following lemma, where we make use of the notation of \Cref{equivalence}.

\begin{lemma}\label{ba}
Let $G=A\ast_K B$, where $A,B$ are GQD on the finite group $K$ and elements $b,b'$,
and assume $A\cong B$. Let $a\coloneqq bb'$.
Then:
\begin{enumerate}[\rm(i)]
\item $K\unlhd G$, and for all $k\in K$ we have
\[
bkb^{-1}=k^{-1}\qquad\text{and}\qquad b'kb'^{-1}=k^{-1}.
\]
\item $a\in C_G(K)$.
\item $ba^m=a^{-m}b$ for all $m\in\Z$.
\item $G=\langle K,a,b\rangle$ and $K\langle a\rangle\unlhd G$.
\item $G=K\langle a\rangle\langle b\rangle$.
\end{enumerate}
\end{lemma}

\begin{proof}
(i) As per \Cref{corclass}, $K$ is a normal subgroup. The second part follows directly from the fact that $A$ and $B$ are generalized quasi-dihedral groups on $K$.

(ii) Follows from \Cref{bb'}.

(iii) Utilizing \Cref{bak=(ka)^{-1}b} that appeared in \Cref{equivalence}, we can infer that $bab^{-1}=a^{-1}$, which leads to $ba^mb^{-1}=a^{-m}$ for any $m\in \Z$.

(iv) Observe that $G=\langle b,b', K\rangle$ and so $G=\langle K,a,b\rangle$. To establish that $K\langle a\rangle $ is a normal subgroup, it suffices to show that $x(ka^i)x^{-1}\in K\langle a\rangle$ for all $x\in \{b^{\pm 1},a^{\pm 1}\}\cup K$. From \Cref{bb'}, it follows that $a$ and elements of $K$ commute, implying that $k'(ka^i)k'^{-1}\in K\langle a\rangle$ and $a(ka^i)a^{-1}\in K\langle a\rangle$. By (iii), we have $ba^ib^{-1}=a^{-i}$. Thus, $bka^ib^{-1}=k^{-1}ba^{i}b^{-1}=k^{-1}a^{-i}$. Therefore, we have shown that $x(ka^i)x^{-1}\in K\langle a\rangle$ for all $x\in \{b^{\pm 1},a^{\pm 1}\}\cup K$.

(v) This is a direct consequence of (iv).
\end{proof}

\begin{coro}\label{normal}
Let $G, A, B, K, a, b$ as in \Cref{ba} and let $G=\langle S\rangle $. If $S\cap K\langle a \rangle \neq\emptyset$, then $\langle S\cap K\langle a \rangle\rangle\unlhd G$. 
\end{coro}

\begin{proof}
Since $K \langle a \rangle$ is an abelian subgroup, by \Cref{ba}(ii) we have $bgb^{-1}=b(ka^i)b^{-1}=k^{-1}a^{-i}=g^{-1}$ for every $g=ka^i \in K \langle a \rangle$.
By~\Cref{ba}(v), we know that $G=K\langle a\rangle \langle b\rangle$ and so we deduce that $\langle S\cap K\langle a \rangle\rangle\unlhd G$, as desired.
\end{proof}

\begin{coro}\label{form}
Let $G, A, B, K, a, b$ as in \Cref{ba}. If $g\in G$, then $g=ka^ib^j$, where $k\in K,i\in \Z$  and $j\in \{0,1\}$.   
\end{coro}
\begin{proof}
By \Cref{ba}(v), we have that $g=ka^ib^j$ for every $g \in G$, where $k\in K,i\in \Z$,  and $j\in \{0,1,2,3\}$ by \Cref{b4=1}.   
If $g=ka^ib^2$, then $b^2\in K$ and so $ab^2=b^2a$(see \Cref{bb'}). This implies that $g=k'a^i$ for $k'=kb^2\in K$.
If $g=ka^ib^3$, then $g=ka^ib^2b$ and so $g=k'a^ib$, where $k'\in K$.
\end{proof}

\begin{lemma}\label{order}
Let $G, A, B, K, a, b$ as in \Cref{ba}. The following statements hold:
\begin{enumerate}
\item $g\in G$ is torsion if and only if $g\in K$ or $g=ka^nb$, where $k\in K$ and $n\in \Z$.
\item $g\in G$ is not torsion if and only if $g=ka^n$, where $k\in K$ and $n\in \Z^*$.
\end{enumerate}    
\end{lemma}
\begin{proof}
Let $g$ be an arbitrary element in $G$.
It follows from \Cref{form} that $Kg$ is either $Ka^n$ or $Ka^nb$.\\
{\textsc{Case 1.}}
Assume that $Kg=Ka^nb$.
Then we have 
\begin{align}
Kg^2&=Ka^nba^nb\label{o}\\
    &=Ka^{n}a^{-n}bb\label{oo}\\
    &=Kb^2\nonumber\\
    &=K,\nonumber
\end{align}
where \Cref{oo} follows from \Cref{ba}(iii).
Since $b^2\in K$, we deduce that $g^2$ lies in $K$ and so $g$ has finite order.\\
{\textsc{Case 2.}}
Let $Kg=Ka^n$.
It remains to show that $G$ is not a torsion element whenever $g \in Ka^n$. Assume to contrary that $(ka^n)^{\ell}=1$, where $n\neq 0$ and so $a^{n\ell}\in K$. This yields a contradiction to the fact that $K$ is finite and $a$ is not a torsion element.    
\end{proof}

We know that very subgroup of a finite dihedral group is either cyclic or dihedral.
It is seeminlgy folklore that this property translates to the infinite dihedral group as well, but we couldn't find a reference for it. 

\begin{theorem}\label{semidihedral subgroup}
Let $G$ be a generalized quasi-dihedral group on $K$ and $b$.
Let $H$ be a subgroup of $G$. Then $H$ is either abelian or generalized quasi-dihedral. 
In particular, every subgroup of an infinite dihedral group is either cyclic or infinite dihedral. 
\end{theorem}

\begin{proof}
If $H$ is a subgroup of $K$, then $H$ is abelian and we are done.
Assume that $H\nsubseteq K$. Then $G=HK$ and so $[H:H\cap K]=[HK:K]=[G:K]=2$.
Since $G=K\sqcup Kb$, there is a $k \in K$ such that $bk\in H$ and so $H=(H\cap K)\sqcup (H\cap K)bk$.
Then for every $x\in K$, we have that $(bk)x(bk)^{-1}=bkxk^{-1}b^{-1}=bxb^{-1}=x^{-1}$.
We show that the order of $bk$ is finite.
Indeed, $(bk)^2=bkbk=bkk^{-1}b$ and so $(bk)^2=b^2$.
Since $G$ is a generalized quasi-dihedral group on $K$ and $b$, we know that the order $b$ is finite.
So $H$ is a generalized quasi-dihedral group on $H\cap K$ and $bk$, as desired.
\end{proof}


\subsection{Short cycles in Cayley graphs of two-ended GQD groups}\label{short cycles}


In this subsection, we prove some crucial facts about $6$- and $4$-cycles in the Cayley graphs of two-ended generalized quasi-dihedral groups. 
For the rest of the subsection, $G, A, B, K, a, b, b'$ are as in \Cref{ba}. Recall that $G = K \langle a \rangle \langle b \rangle$ by \Cref{ba}(v), where $a = bb'$.

\begin{theorem}\label{6-cycle}
Let $G=\langle S\rangle$ and let $s_1,s_2,s_3$ be three distinct elements in $G\sm K\langle a \rangle$. Then $s_1s_2s_3=s_3s_2s_1$. If $s_1,s_2,s_3 \in S\sm K\langle a \rangle$ in particular, then for every $g\in G$ the sequence of vertices $g,gs_1,gs_1s_2,gs_1s_2s_3,gs_3s_2,gs_3,g$ is a $6$-cycle in $\cay(G;S)$.
\end{theorem}

\begin{proof}
By~\Cref{form}, we know that that the elements of $G\sm K\langle a \rangle$ are of the form $ka^ib$.
So, let $s_1=k_{\ell}a^{\ell}b$, $s_2=k_{i}a^{i}b$ and $s_3=k_{j}a^{j}b$.
The following computation follows from~\Cref{semidihedral} and~\Cref{form} and completes the proof:
\begin{align*}
s_1s_2s_3=&(k_{\ell}a^{\ell}b)(k_{i}a^{i}b)(k_{j}a^{j}b) \\
  =&(k_{\ell}a^{\ell}b)k_{i}a^{i}k_j^{-1}a^{-j}b^2 \tag{push the second $b$ to the right}\\
   =&k_{\ell}a^{\ell}bk_j^{-1}a^{-j}k_ia^{i}b^2 \\
    =&k_ja^{j}k_{\ell}a^{\ell}bk_ia^{i}b^2 \tag{push the first $b$ to the middle and $k_ja^j$ to the left}\\
     =&(k_ja^{j}b)k_{\ell}^{-1}a^{-\ell}k_ia^{i}b^2\\
      =&s_3k_ia^{i}k_{\ell}^{-1}a^{-\ell}b^2\\
       =&s_3k_{i}a^{i}bk_{\ell}a^{\ell}b \tag{push a $b$ to the left}\\
        =&s_3s_2s_1. \hspace{8.9cm}\qed
\end{align*}
\renewcommand{\qedsymbol}{}
\end{proof}




\begin{theorem}\label{4-cycle}
Let $G=\langle S\rangle$ and let $s_1\in G\sm K\langle a\rangle$ and $s_2\in K\langle a\rangle$.
Then $s_1s_2 = s_2^{-1}s_1$.
In particular, if $S$ is symmetric and $s_1\in S\sm K\langle a\rangle$ and $s_2\in S\cap K\langle a\rangle$,
then for every $g\in G$ the vertices
$g,\; gs_1,\; gs_1s_2,\; gs_2^{-1},\; g$
form a $4$-cycle in $\cay(G;S)$.
\end{theorem}

\begin{proof}
Write $s_1=x b$ with $x\in K\langle a\rangle$ and $s_2=y$ with $y\in K\langle a\rangle$.
Since $b y b^{-1}=y^{-1}$, we have $by=y^{-1}b$ and hence
\[
s_1s_2=(xb)y = x(by)=x(y^{-1}b)=y^{-1}(xb)=s_2^{-1}s_1,
\]
using that $K\langle a\rangle$ is abelian.
For the $4$-cycle, we note that
\[
(gs_2^{-1})s_1 = g(s_2^{-1}s_1)=g(s_1s_2)=gs_1s_2,
\]
so indeed $gs_1s_2$ is adjacent to $gs_2^{-1}$ by an $s_1$-edge.
\end{proof}


\section{Grids, Walls and Cylinders}\label{sec:walls}

As we already mentioned in the Introduction, the proof of~\Cref{main} is based on extracting grids and cylindrical walls as spanning subgraphs. In this short section, we discuss the Hamiltonicity of the grid-like structures.

The \emph{Cartesian product} $G\Box H$  of two graphs $G$ and $H$ is the graph with vertex set $V(G)\times V(H)$, 
where vertices $(a,x)$ and $(b,y)$ are adjacent whenever $ab\in E(G)$ and $x=y$, or $a=b$
and $xy\in E(H)$.

Let $P_k$ denote the path graph on $k$ vertices and $D$ the double ray graph. 
The \emph{(two-ended) infinite grid} $\mathcal G_k$ of height $k$ is the graph $P_k \Box D$. 
The next lemma provides a way to extract a spanning grid and, consequently, a Hamiltonian double ray from a graph.

\begin{lemma}{\rm{\cite[Lemma 3.9]{miraftab2017cycles}}}
\label{spanning grid}
Let $G,H$ be two-ended locally finite graphs and let $G_1,\ldots, G_n$ be disjoint subgraphs of $G$ such that 
\begin{enumerate}[{\rm(i)}]
\item $\bigsqcup_{i=1}^n V(G_i) = V(G)$, and  
\item every $G_i$ contains a spanning subgraph $H_i$, which is isomorphic to~$H$ by means of an isomorphism~$\phi_i: H \rightarrow H_i$, and 
\item for every ${i < n}$ and every~${v \in H_i}$ there is an edge between~$v$ and~${\phi_{i+1} \circ  \phi^{-1}_i(v)}$.
\end{enumerate}
Let $R$ be a spanning double ray of $H$. Then $G$ contains a spanning infinite grid of height $n$. In particular, it contains a Hamiltonian double ray and a Hamiltonian circle\footnote{See page 17 for the definition of circle.}.
\end{lemma}




The \emph{wall graph} $W_k$ of height $k$ is the graph with vertex set $V(W_k)=\Z \times \Z_k$ and edge set $E(W_k)=E_1 \cup E_2 \cup E_3$, where 
\begin{align*}
    E_1&=\{\{(n,m),(n+1,m)\} \mid n\in \Z, m=0,\ldots,k-1\}, \\
    E_2&=\{\{(2n,m),(2n,m+1)\} \mid n\in \Z, m\equiv 0 \ (\rm{mod} 2)\}, \\
    E_3&=\{\{(2n+1,m),(2n+1,m+1)\} \mid n\in \Z, m\equiv 1 \ (\rm{mod} 2)\}.
\end{align*}
For $k+l$ even, the \emph{twisted cubic cylinder} $\overline{W}_{k,l}$ of height $k$ and twist $l\in \Z$ is the graph with vertex set $V(\W_{k,l})=V(W_k)=\Z \times \Z_k$ and edge set $E(\W_{k,l})=E(W_k) \cup E'_{l}$, where
\begin{itemize}

\item $E'_{l}=\{\{(2n+1,k-1),(2n+1+l,0)\} \mid n\in \Z\}$, when $k$ is even,
\item $E'_{l}=\{\{(2n,k-1),(2n+l,0)\} \mid n\in \Z\}$, when $k$ is odd.

\end{itemize}

\noindent
The edges in $E_2 \cup E_3$ are called \emph{straight} and the edges in $E'_{l}$ are called \emph{twisted}. The \emph{$m$-th row} $R_m$ of $W_{k}$ and $\W_{k,l}$ is the double ray induced by $\Z \times \{m\}$. 
For $j\geq 1$ and $i\in \Z$, the \emph{$i$-th block} $B_{i,2j}$ of length $2j$ of $\W_{k,l}$ is the graph induced by the vertices
\begin{itemize}
    \item  $(2i+1+n,m)$, $n=0,\ldots,2j-1$, $m\in \Z_k$, when $k$ is even. The \emph{$i$-th snake} $S_{i,2j}$ of length $2j$ is then the unique Hamiltonian path from $(2i+1,0)$ to $(2i+1,k-1)$ in $B_{i,2j}$.
    \item  $(2i+1-n,m)$, $n=0,\ldots,2j-1$, $m\in \Z_k$, when $k$ is odd. The \emph{$i$-th snake} $S_{i,2j}$ of length $2j$ is then the unique Hamiltonian path from $(2i+1,0)$ to $(2i+2-2j,k-1)$ in $B_{i,2j}$.
\end{itemize}
\begin{figure}[H]
    \centering
    \includegraphics[scale=0.8]{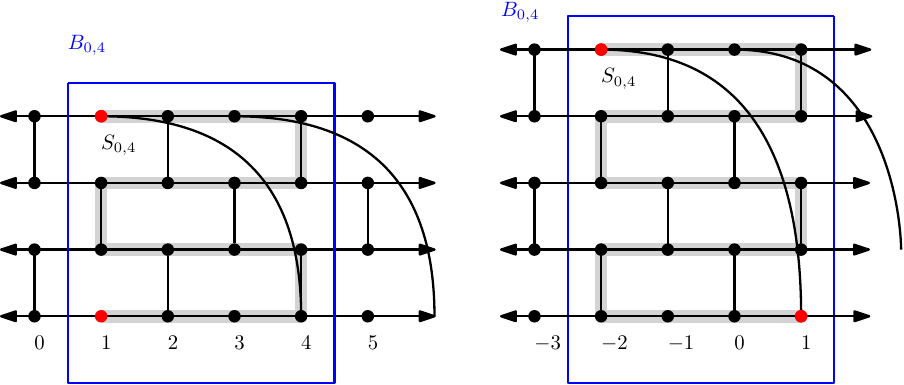}
    \caption{The gray path in the left picture depicts $S_{0,4}$ in $\overline{W}_{4,4}$ and the gray path in the right picture depicts $S_{0,4}$ in $\overline{W}_{5,3}$.}
    \label{fig:my_label}
\end{figure}
In particular, the \emph{$i$-th column} $Q_i$ of $\W_{k,l}$ is the path $S_{i,2}$.
The \emph{$i$-th staircase} $\Gamma_i$ of $\W_{k,l}$ is the unique path of length $2k-1$ from $(2i+1,0)$ to $(2i+1+k,k-1)$.\\

\begin{figure}[H]\label{twisted cylinder}

\includegraphics[scale=0.9]{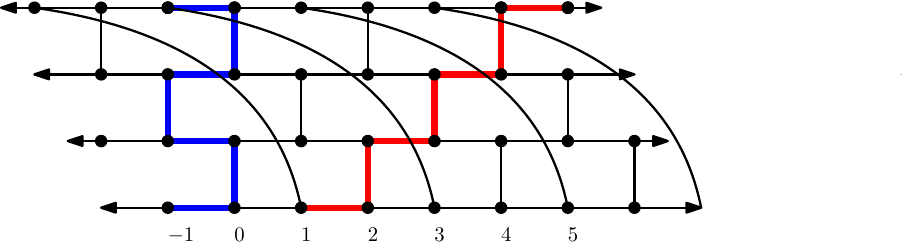}
\caption{A column (blue) and a staircase (red) in the twisted cubic cylinder $\W_{4,4}$.}
\end{figure}

The following lemma captures a surprising property of twisted cylinders: stretching a given twisted cylinder along carefully chosen double rays can transform it into a twisted cylinder with different parameters (see~\Cref{cylinder isomorphism proof}).

\begin{lemma}\label{cylinder isomorphism}
For any $k\geq 2$ and $l \in \N$ with $k+l$ even, the twisted cylinders $\W_{k,l}$ and $\W_{\frac{k+l}{2}, \frac{3k-l}{2}}$ are isomorphic.
\end{lemma}

\begin{proof}
For $r=0,\ldots,\frac{k+l}{2}-1$, let $D_r$ be the double ray obtained by the concatenation of the all staircases $\Gamma_{i\frac{k+l}{2}+r}$ along twisted edges. We observe that $\W_{k,l}$ is isomorphic to $\W_{\frac{k+l}{2}, \frac{3k-l}{2}}$ with rows $D_0,\ldots,D_{\frac{k+l}{2}-1}$.
\end{proof}

\begin{figure}[!h]
\begin{center}
\begin{tikzpicture}
\begin{scope}[scale=0.85]

\draw[<->,>=latex'] (-0.5,1)--(8.5,1);
\draw[<->,>=latex'] (0,0)--(9,0);
\draw[<->,>=latex'] (0.3,-1)--(9.5,-1);
\draw[<->,>=latex'] (0.6,-2)--(10,-2);

\draw (1,1)--(1,0);
\draw (3,1)--(3,0);
\draw (5,1)--(5,0);
\draw (7,1)--(7,0);
\draw (2,0)--(2,-1);
\draw (4,0)--(4,-1);
\draw (6,0)--(6,-1);
\draw (8,0)--(8,-1);
\draw (9,-2)--(9,-1);
\draw (7,-2)--(7,-1);
\draw (5,-2)--(5,-1);
\draw (3,-2)--(3,-1);
\draw (1,-2)--(1,-1);

\begin{scope}[ultra thick, draw=blue]
\draw (0,-2)--(1,-2);
\draw (1,-1)--(2,-1);
\draw (3,0)--(2,0);
\draw (3,1)--(4,1);
\draw (2,0)--(2,-1);
\draw (1,-2)--(1,-1);
\draw (3,1)--(3,0);
\end{scope}

\begin{scope}[shift={(6,0)}, ultra thick, draw=blue]
\draw (0,-2)--(1,-2);
\draw (1,-1)--(2,-1);
\draw (3,0)--(2,0);
\draw (3,1)--(4,1);
\draw (2,0)--(2,-1);
\draw (1,-2)--(1,-1);
\draw (3,1)--(3,0);
\end{scope}

\begin{scope}[shift={(2,0)}, ultra thick, draw=orange]
\draw (0,-2)--(1,-2);
\draw (1,-1)--(2,-1);
\draw (3,0)--(2,0);
\draw (3,1)--(4,1);
\draw (2,0)--(2,-1);
\draw (1,-2)--(1,-1);
\draw (3,1)--(3,0);
\end{scope}

\begin{scope}[shift={(8,0)}, ultra thick, draw=orange]
\draw (0,-2)--(1,-2);
\draw (1,-1)--(2,-1);
\draw (3,0)--(2,0);
\draw (3,1)--(4,1);
\draw (2,0)--(2,-1);
\draw (1,-2)--(1,-1);
\draw (3,1)--(3,0);
\end{scope}

\begin{scope}[shift={(4,0)}, ultra thick, draw=green]
\draw (0,-2)--(1,-2);
\draw (1,-1)--(2,-1);
\draw (3,0)--(2,0);
\draw (3,1)--(4,1);
\draw (2,0)--(2,-1);
\draw (1,-2)--(1,-1);
\draw (3,1)--(3,0);
\end{scope}

\begin{scope}[shift={(-2,0)}, draw=green]
\draw (0,-2)--(1,-2);
\draw (1,-1)--(2,-1);
\draw (3,0)--(2,0);
\draw (3,1)--(4,1);
\draw (2,0)--(2,-1);
\draw (1,-2)--(1,-1);
\draw (3,1)--(3,0);
\end{scope}

\draw [bend left, ultra thick, orange] (0,1) to (2,-2);
\draw [bend left, ultra thick, green] (2,1) to (4,-2);
\draw [bend left, ultra thick, blue] (4,1) to (6,-2);
\draw [bend left, ultra thick, green] (8,1) to (9.8,-1.8);
\draw [bend left, ultra thick, orange] (6,1) to (8,-2);

\begin{scope}[very thick, draw=red]
\draw[decorate, decoration=snake] (2,1)--(3,1);
\draw[decorate, decoration=snake] (1,0)--(2,0);
\draw[decorate, decoration=snake] (0,-1)--(1,-1);
\draw[decorate, decoration=snake] (-1,-2)--(0,-2);
\end{scope}

\begin{scope}[shift={(6,0)}, very thick, draw=red]
\draw[decorate, decoration=snake] (2,1)--(3,1);
\draw[decorate, decoration=snake] (1,0)--(2,0);
\draw[decorate, decoration=snake] (0,-1)--(1,-1);
\draw[decorate, decoration=snake] (-1,-2)--(0,-2);
\end{scope}

\foreach \i in {1,2,...,8}{
\draw (\i,1) node [circle,fill, inner sep=2pt] {};
\draw (\i,0) node [circle,fill, inner sep=2pt] {};
\draw (\i,-1) node [circle,fill, inner sep=2pt] {};
\draw (\i,-2) node [circle,fill, inner sep=2pt] {};
\draw (0,1) node [circle,fill, inner sep=2pt] {};
\draw (9,-1) node [circle,fill, inner sep=2pt] {};
\draw (9,-2) node [circle,fill, inner sep=2pt] {};
}

\begin{scope}[shift={(0.5,0.5)}]
\end{scope}
\end{scope}
\end{tikzpicture}
\end{center}
\caption{The double rays of $\W_{4,2}$ that serve as rows of $\W_{3,5}$. The red wavy edges correspond to twisted edges between the first(blue) and last(green) double ray of $\W_{3,5}$.}
\label{cylinder isomorphism proof}
\end{figure}

\begin{lemma}\label{cylinder double ray}
For any $k\geq 2$ and $l \in \N$ with $k+l$ even, the twisted cubic cylinder $\W_{k,l}$ contains a Hamiltonian double ray.
\end{lemma}

\begin{proof}
If $k$ and $l$ are even $\geq 2$, we obtain a Hamiltonian double ray by concatenating all snakes $S_{il,l}$, $i \in \Z$ along twisted edges.

If $k$ and $l$ are odd $\geq 3$, let $l_1,l_2$ be positive even numbers such that $l_1+l_2=l+1$. We then obtain a Hamiltonian double ray by concatenating all snakes $S_{il,l_1}$ and $S_{il-l_1,l_2}$, $i \in \Z$ in an alternating fashion along twisted edges.

It remains to check the cases when $l=0$ (and $k$ necessarily even) or $l=1$ (and $k$ odd). For $r=0,\ldots,k+l-1$, let $D_r$ be the double ray obtained by the concatenation of the all staircases $\Gamma_{i\frac{k+l}{2}+r}$ along twisted edges. 
We conclude the proof by invoking~\Cref{cylinder isomorphism} and noticing that $\frac{3k-l}{2} \geq 2$.
\end{proof}

\begin{figure}[!h]
    \centering
    \includegraphics[scale=0.8]{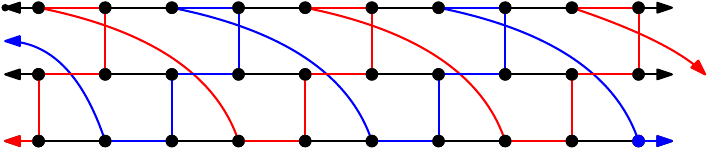}
    \caption{The union of the blue and red double rays covers all vertices of $\W_{3,3}$}
    \label{fig:enter-label}
\end{figure}

\begin{lemma}\label{cylinder circle}
For any $k\geq 2$ and $l \in \N$, the twisted cubic cylinder $\W_{k,l}$ contains a disjoint union of two double-rays which together span $\W_{k,l}$.
\end{lemma}

\begin{proof}
We can assume that $k\geq 3$, as otherwise the result is trivial for $k=2$. Moreover, we reduce the cases where $l=0,1,2$ to the next (general) cases by invoking~\Cref{cylinder isomorphism}, where as the case $k=l=3$ follows from~\Cref{fig:enter-label}.

If $k$ and $l$ are even $\geq 4$, let $l_1, l_2$ be positive even integers such that $l_1+l_2=l$. We let $D_1$ be the double ray obtained by the concatenation of all snakes $S_{il,l_1}$ and $D_2$ the double ray obtained by the concatenation of all snakes $S_{il+l_1+1,l_2}$, $i \in \Z$ along twisted edges as the two disjoint double rays that span a Hamiltonian circle of $\W_{k,l}$.

If $k$ and $l$ are odd $\geq 5$, let $l_1,l_2,l_3$ be positive even numbers such that $l_1+l_2+l_3=l+1$. As before, we define the double ray $D_1$ as the concatenation of all snakes $S_{il,l_1}$ and $S_{il-l_1,l_2}$, $i \in \Z$ in an alternating fashion along twisted edges and $D_2$ as the concatenation of all $S_{il-l_1-l_2,l_3}$, $i \in \Z$. The double rays $D_1$ and $D_2$ are then disjoint and span the whole graph.
\end{proof}

\section{Proof of the main Theorem}

Let us quickly remind the following basic fact.

\begin{lemma}[Modular law {\rm\cite[1.6.15]{scott}}]\label{modular}
Let $G$ be a group and let $H,K,L\le G$ be subgroups with $H\le L$.
Then
\[
L\cap HK = H(L\cap K).
\]
\end{lemma}

\begin{lemma}\label{genindex2}
Let $G=\langle S\rangle$ be a group with a subgroup $H$ of index $2$.
Assume $S$ is symmetric and $S\cap H=\emptyset$.
Then for any $s\in S$, the set $sS$ generates $H$.
\end{lemma}

\begin{proof}
Let $h \in H$ be an arbitrary element of $H$.
Then $h = x_1 \cdots x_t$, where each $x_i \in S$ for $i = 1, \ldots, t$.
We observe that $t$ must be even so that $h$ belongs to $H$.
This shows that $S^2$ generates $H$.
Finally, we prove that every element $s_i s_j$ of $S^2$ can be generated by $sS$ by noting that $s_i s_j = s_i s^{-1} s s_j$, and $s s_i^{-1} \in sS$.
This completes the proof.
\end{proof}
\begin{remark}\label{t}
    There is an element $a\in  D_\infty$ such that $[G:\langle a\rangle ]=2$ and moreover $\langle a,t\rangle =D_{\infty}$ for every $t\in D_{\infty}\sm \langle a \rangle $.
\end{remark}
We are ready to prove our main result. 
We will extensively use~\Cref{semidihedral} and~\Cref{ba} for the calculations throughout the proof, so we will mostly not specifically refer to them each time we use them. 

\begin{proof}[\textbf{Proof of~\Cref{main}}]
Let $G=A\ast_K B = \langle S\rangle$, where  $A$ and $B$ are generalized quasi-dihedral on $K$, $b$ and $K$, $b'$ and $A\cong B$.
It follows from \Cref{b2} that  $b^2=b'^2$.
Then $A=K\sqcup Kb$, $B=K\sqcup Kb'$ as in~\Cref{ba} and recall that $a=bb'$.

We apply induction on the number of generators in $S$.
For the base case that $|S|=2$, we immediately see that $\cay(G;S)$ is a double ray (and it is straightforward to show that $G\cong D_{\infty}$).

Now, let $|S|\geq 3$. We split the proof of the induction step into two cases.\\
{\textsc{Case 1.}}
We first assume that $S\cap K\langle a\rangle=\emptyset$. 
We claim the following:
\begin{claim}
There is an~$s\in S$ such that~$H\coloneqq\langle S' \rangle$ is infinite, where~$S':=S\sm \{s^{\pm 1}\}$.
\end{claim}
\begin{proof}
Assume to the contrary that $H$ is finite for every $S\sm \{s^{\pm 1}\}$, and so every element of $S$ is torsion.  
By~\Cref{order}, every generator in $S$ has the form $ka^ib$, where $k\in K$ and $i\in \Z$.
Notice that it can't be that $i=j$ for every $s_1=k_1a^ib, s_2=k_2a^jb \in S$, as otherwise $S$ does not generate $G$. 
Because $\langle s_1,s_2\rangle \subseteq K\langle a^ib\rangle $
Thus, there exist $s_1=ka^{i}b, s_2=k'a^{j}b\in S$, where $j\neq i$. Recall that $|S'| \geq 2$ and let $s \in S \setminus \{ s_1^{\pm 1}, s_2^{\pm 1}\}$.
Then $H\coloneqq\langle S' \rangle$ is infinite, where~$S':=S\sm \{s^{\pm 1}\}$, as $s_1s_2=k_1k_2^{-1}b^2a^{i-j} \in H \cap K\langle a \rangle$ has infinite order.
\end{proof}

Since $K\unlhd HK$, we deduce that $(H\cap K)\unlhd H$, hence $ H/(H\cap K)\cong HK/K\leq G/K\cong\Dinf$.
It follows from \Cref{semidihedral subgroup} that every subgroup of $D_{\infty}$ is isomorphic to either a cyclic group or $D_{\infty}$.
But since $H$ is infinite and $S\cap K\langle  a\rangle=\emptyset$, there are elements $s_1=k_1a^ib$ and $s_2=k_2a^jb$, where $i\neq j$ and so $s_1s_2\neq s_2s_1$.
Thus $H/(H\cap K)\cong D_{\infty}$.
It follows from \Cref{order} and \Cref{t} that there is an element $k_{\ell}a^{\ell}\in K\langle a \rangle $ such that 
\begin{equation}\label{[H:H']}
2=\left[\frac{H}{H\cap K}:\langle k_{\ell}a^{\ell}(H\cap K)\rangle \right]=[H:(H\cap K)\langle k_{\ell}a^{\ell}\rangle].
\end{equation}
We note that $S\cap K\langle a\rangle =\emptyset$.
By \Cref{t}, $\langle k_{\ell}a^{\ell}(H\cap K), t(H\cap K)\rangle$ is $H/(H\cap K)$ for some $t\in S'$.
Next it follows that $H=(H\cap K)\langle k_{\ell}a^{\ell},  t\rangle$, where $t\in S'$ and for some $k_{\ell}a^{\ell}\in H$.

We now state the following calculation, some steps of which we explain afterwards:
\begin{align}
 K\langle a \rangle=&\langle S',s\rangle \cap K\langle a\rangle\nonumber\\
 =&\langle H,s\rangle \cap K\langle a \rangle\nonumber\\
  =&\langle (H\cap K)\langle k_{\ell}a^{\ell},t\rangle,s\rangle \cap K\langle a\rangle\nonumber\\
  =&\left((H\cap K)\langle k_{\ell}a^{\ell}\rangle\langle t,s\rangle\right) \cap K\langle a\rangle \label{eq1}\\
    =&\left((H\cap K)\langle k_{\ell}a^{\ell}\rangle\right)\left(\langle t,s\rangle\cap K\langle a \rangle\right) \tag{by~\Cref{modular}}\\
        =&\left((H\cap K)\langle k_{\ell}a^{\ell}\rangle\right)\langle ts \rangle \label{eq2}
\end{align}

\noindent
Recall that $S\cap K\langle a\rangle =\emptyset$.
Let $s=k_ia^ib$ and $t=k_ja^jb$. 
For~(\Cref{eq1}), we first note 
that $(H\cap K) s=s (H\cap K)$, as $k\in H\cap K$ and so $ks=kk_ia^ib=k_ia^ibk^{-1}=sk^{-1}$.
In addition, we have
\[
 (k_{\ell}a^{\ell})s=k_{\ell}a^{\ell}k_ia^ib
 =k_ia^ik_{\ell}a^{\ell}b
  =k_ia^ibk_{\ell}^{-1}a^{-\ell}
  =s(k_{\ell}a^{\ell})^{-1},
\]
So $ks=sk^{-1}$ and $(k_{\ell}a^{\ell})s=s(k_{\ell}a^{\ell})^{-1}$. 
Similarly, one can show that $kt=tk^{-1}$and $(k_{\ell}a^{\ell})t=t(k_{\ell}a^{\ell})^{-1}$ as well. 
Thus $(H\cap K)\langle s,t\rangle=\langle s,t\rangle(H\cap K)$ and  $\langle k_{\ell}a^{\ell}\rangle\langle t,s\rangle=\langle t,s\rangle\langle k_{\ell}a^{\ell}\rangle $.
Hence $(H\cap K)\langle k_{\ell}a^{\ell}\rangle\langle t,s\rangle$ is a subgroup of $G$, which implies that
$$(H\cap K)\langle k_{\ell}a^{\ell}\rangle\langle t,s\rangle=\langle (H\cap K)\langle k_{\ell}a^{\ell},t\rangle,s\rangle.$$

\noindent Let~$H'\coloneqq (H\cap K)\langle k_{\ell}a^{\ell}\rangle$. 
We have proven so far that~$K\langle a \rangle=H'\left(\langle t,s\rangle\cap K\langle a \rangle\right)$,
so it remains to show~(\Cref{eq2}), or that $\langle t,s\rangle\cap K\langle a \rangle=\langle ts \rangle$.\\
It is not hard to see that $ \langle t,s\rangle K\langle a \rangle=G$, where $s=k_ia^ib$ and $t=k_ja^jb$.
By \Cref{[G:Ka]=2}, we know that $[G:K\langle a \rangle]=2$ and so $G=K\langle a\rangle \sqcup K\langle a \rangle s$.
Moreover, we have  $$2=[G:K\langle a \rangle]=[\langle t,s\rangle K\langle a \rangle:K\langle a \rangle]=[\langle t,s \rangle:\langle t,s\rangle \cap K\langle a \rangle].$$
We can now apply~\Cref{genindex2} and infer that $\{ts,ts^{-1}, t^2\}$ generates  $\langle t,s\rangle \cap K\langle a \rangle$.
\begin{claim}
The cosets $H'(ts)^{\ell}ts^{-1}$ and~$H'(ts)^{\ell}t^{2}$ can be expressed as~$H'(ts)^{\ell'}$, where~$\ell,\ell'\in \Z$ and $H'=(H\cap K)\langle k_{\ell}a^{\ell}\rangle$.
\end{claim}
\begin{proof}
Observe that $ts=ts^{-1}s^2=ts^{-1}b^2=ts^{-1}t^2$.
Moreover, notice that $t^2$, $ts$ and $ts^{-1}$ lie in $K \langle a \rangle$, hence they pairwise commute.
Since $t^2\in H\cap K\subseteq H'$, we deduce that
\[
H'(ts)^{\ell}t^2=H't^2(ts)^{\ell}=H'(ts)^{\ell}
\]
and
\[
    H'(ts)^{\ell}ts^{-1}=H't^2(ts)^{\ell}ts^{-1}= H' (ts)^{\ell}ts^{-1}t^2=H' (ts)^{\ell+1}.
    \]
This proves the claim.
\end{proof}
Thus, we have shown that every coset of $H'$ in $K\langle a\rangle$ is of the form $H'(ts)^{\ell}$.
Next, we claim the following:
\begin{claim}
    $H'$ has finite index in $K\langle a\rangle$.
\end{claim}
\begin{proof}
    We note that $ts=k_ja^jbk_ia^ib=ka^{j-i}$ for some $k\in K$.
    In addition we already know that $H'=(H\cap K)\langle k_{\ell}a^{\ell}\rangle$.
    Since $(ts)^n=k^na^{(j-i)n}$, we deduce that there is $n\in \N$ such that $(ts)^n\in H'$ which shows that the index is finite.
\end{proof}
Therefore we have
\[
K\langle a\rangle=H'\sqcup H'(ts)\sqcup\cdots \sqcup H'(ts)^m.
\]
Recall that $G=K\langle a\rangle \sqcup K\langle a \rangle s$.
As a result, the cosets of $H'$ partition $G$ as follows:
\begin{align}
 G=&H'\sqcup H'(ts)\sqcup\cdots \sqcup H'(ts)^m\sqcup\\
 &\phantom{{}=1}H't\sqcup H'(ts)t\sqcup\cdots \sqcup H'(ts)^mt.\nonumber
 \end{align}

\begin{figure}[!h]\label{twistedcylinder}

\begin{center}
\begin{tikzpicture}[scale=0.7]
\foreach \i in {1,2,...,8}{
\draw[thick] (\i,0)--(\i,1);
\draw[thick,blue] (\i,1)--(\i,2.5);
\draw[thick,blue] (\i,3.5)--(\i,5);
\draw[thick] (\i,2.5)--(\i,3.5);
\draw[thick] (\i,5)--(\i,6);
}
\foreach \i in {1,2,...,7}{
\draw[thick] (\i,1)--(\i+1,0);
\draw[thick] (\i,3.5)--(\i+1,2.5);
\draw[thick] (\i,6)--(\i+1,5);
}
\foreach \i in {1,2,...,8}{
\draw (\i,1) node [circle,fill, inner sep=2pt] {};
\draw (\i,0) node [circle,fill, inner sep=2pt] {};
\draw (\i,2.5) node [circle,fill, inner sep=2pt] {};
\draw (\i,3.5) node [circle,fill, inner sep=2pt] {};
\draw (\i,5) node [circle,fill, inner sep=2pt] {};
\draw (\i,6) node [circle,fill, inner sep=2pt] {};
}
\draw (9,0) node {$H'$};
\draw (9,1) node {$H't$};
\draw (9,2.5) node {$H'ts$};
\draw (9,3.5) node {$H'tst$};
\draw (9,5) node {$H'tsts$};
\draw (9,6) node {$H'tstst$};

\begin{scope}[shift={(0.5,0.5)}]
\end{scope}
\end{tikzpicture}
\end{center}
\caption{The partition of $G$ according to the cosets of $H'$.}
\end{figure}
 
Finally, we need the following observation. It is an easy induction on $\ell$ to see that for every  $1\leq \ell \leq m$ and $t_1,\ldots,t_{\ell} \in S'$ (not necessarily distinct) we have that
\begin{align}
    H'(t_1s)(t_2s)\ldots (t_{\ell}s)=&H'(ts)^{\ell}, \label{ts}\\
    H'(t_1s)(t_2s)\ldots (t_{\ell-1}s)t_{\ell}=&H'(ts)^{\ell-1}t \label{tst}.
\end{align}

We are ready to extract a spanning twisted cubic cylinder from $\cay(G;S)$.
By \Cref{[H:H']} we know that $[H:H']=2$. Now, since $t\notin H'$, we conclude that  $H=H' \cup H't$ has a spanning double ray $R_0$, which must necessarily alternate between the two cosets of $H'$ in $H$ by the fact that $S' \subseteq H't$.  Split $R$ into subpaths of length two between vertices of $H't$ and let $P=\{g,gt_1,gt_1t_2\}$ be an arbitrary such subpath, where $g \in H't$. By~\Cref{6-cycle}, there is a $6$-cycle $C=\{g,gt_1,gt_1t_2,gt_1t_2s,gst_2,gs\}$. In other words, $C$ is obtained by connecting the ends of $P$ and $P'=\{gst_1t_2,gst_2,gs\}$ with two edges labeled with $s$. Clearly, $gt_1t_2s,gs \in H't$ and by~(\Cref{tst}) we have that $gst_2\in H'tst$. It follows that the concatenation of all such paths $P'$ of length two gives rise to a spanning double ray $R_1$ of $H'ts \cup H'tst $ alternating between vertices of $H'ts$ and $H'tst$, where the vertices of $H't$ in $R_1$ are connected by an $s$-edge to the vertices of $H'ts$ in $R_2$. 

Using exactly the same method, we obtain inductively for every $\ell \in  \Z_{m+1}$ that $H'(ts)^{\ell} \cup H'(ts)^{\ell}t$ has a spanning double ray $R_{\ell}$ alternating between  $H'(ts)^{\ell}$ and $H'(ts)^{\ell}t$, where the vertices of $H'(ts)^{\ell}t$ in $R_{\ell}$ are connected by an $s$-edge to the vertices of $H'(ts)^{\ell+1}$ in $R_{\ell+1}$. 
\Cref{6-cycle} ascertains that this gives rise to a twisted cubic cylinder with rows $R_0,R_1,\ldots,R_m$.
We conclude by~\Cref{cylinder double ray} and ~\Cref{cylinder circle} that $\cay(G;S)$ contains a Hamiltonian double ray.\\

{\textsc{Case 2.}}
Suppose that $S\cap K\langle a \rangle \neq \emptyset$. Let $S_1:=S\sm K\langle a \rangle$ (notice that $S_1\neq \emptyset$) and $S_2:=S\sm S_1$.
Let $H'$ be the subgroup generated by $S_2$.\\
{\bf Subcase i:} If $H'$ is finite, then $H'\subseteq K$.
We note that $G/H'\cong A/H'\ast_{K/H'}A/H'$ and so $G/H'$ is generalized quasi-dihedral.
By the induction hypothesis, the Cayley graph of $G/H'$ with respect with $S_1H'$ has a Hamiltonian double ray $\mathcal R$
\[
\ldots,H'x_{-2},H'x_{-1},H',H'x_1,H'x_2,\ldots,
\]
where $x_i\in S_1$ for each $i\in \mathbb Z\sm\{0\}$.
On the other hand, $\cay(H',S_2)$ has a Hamiltonian path.
Since by~\Cref{4-cycle} we have 
$xk=k^{-1}x$ and $kx=xk^{-1}$ for every $x\in S_1$ and $k\in K$, we invoke~\Cref{spanning grid} and conclude that $G$ has a Hamiltonian double ray.\\
{\bf Subcase ii:} If $H'$ is infinite, then $[G:H']$ is finite.
We note that $H'$ is an abelian group and so it follows from \cite[Theorem 1] {nash1959abelian} that $\cay(H',S_2)$ contains a Hamiltonian double ray $R$.
By~\Cref{normal}, we know that $H'$ is normal. Moreover, $S_1H'$ generates the quotient $G/H'$.
On the other hand, we know by~\Cref{ba} that $\langle KH', aH'\rangle=\langle K,a \rangle / H'=K \langle a \rangle / H'$ is an abelian subgroup of $G/H'$.
Furthermore, we have that $[G/H':K\langle a \rangle / H']=[G:K \langle a \rangle]=2$.

Since $bH'gH'=g^{-1}H'bH'$ for every $g \in K\langle a \rangle / H'$ and $[G:H']$ is finite, we deduce that the quotient $G/H'$ is a generalized quasi-dihedral group on $K\langle a \rangle / H'$ and $bH'$.

It follows from~\Cref{semiham} that $\cay(G/H',S_1H')$ has a Hamiltonian path
$$g_1H'=H',g_2 H',\ldots,g_nH'.$$
Observe that $g_iH'$ contains a Hamiltonian double ray $g_iR$ for every $1\leq i \leq n$ in $\cay(G;S)$. 
Moreover, the vertices of $g_iR$ are connected to the vertices of $g_{i+1}R=(g_iR)s_i$ with a perfect matching whose edges are labeled with the same generator $s_i \in S_1$. 
In combination with~\Cref{4-cycle}, this implies that conditions (ii) and (iii) of~\Cref{spanning grid} hold and we are done.
\end{proof}

\begin{section}{Final Remarks}

In the final section, we illustrate different approaches to continue the study of Hamiltonicity of generalized quasi-dihedral groups.

\subsection{Hamiltonian circles}
While there have been many attempts to extend the definition of hamiltonian cycles for infinite graphs, there is no single method that generalizes all theorems of finite Hamiltonicity to locally finite graphs.

Here we follow the topolgical approach introduced in \cite{RDsBanffSurveyII,RDsBanffSurveyI,diestelgraph}.
More specifically, the infinite cycles of a graph $G$ are defined as the circles(homeomorphic image of $S^1$) in the Freudenthal compactification $|G|$, where $|G|$ denotes the graph $G$ endowed by $1$-complex topology. 
A homeomorphic image of $[0,1]$ in $|G|$ is an \emph{arc} in $G$.
A circle in $|G|$ is a \emph{Hamiltonian circle} if it contains all vertices of $G$.
When $G$ is two-ended, there is a nice combinatorial description for Hamiltonian circles, see the following lemma.
\begin{lemma}{\rm\cite[Theorem 2.5]{RDsBanffSurveyI}}\label{What is HC}
If $\G$ is a locally finite, two-ended graph and $C$ is a disjoint union of two double-rays
which together span $\G$, each of which contains a ray to both ends of the graph, then
$C$ is a Hamiltonian circle.
\end{lemma}

By the proof of~\Cref{main} and~\Cref{spanning grid} as well as~\Cref{cylinder circle}, we have the following corollary.

\begin{coro}
Let $G=\langle S\rangle$ be a two-ended generalized quasi-dihedral group.
If $\cay(G;S)$ has degree at least three, it contains a Hamiltonian circle.
\end{coro}

We believe that the same statements of Hamiltonicity go through when $G$ is a one-ended generalized quasi-dihedral group, too. We note that in that case the subrgroup $H$ of $G$ of index $2$ is an abelian one-ended group, hence $H$ is isomorphic to $\mathbb Z^k$ for some $k\geq 2$.

\begin{conj}
Let $G=\langle S\rangle$ be a one-ended generalized quasi-dihedral group.
Then $\cay(G;S)$ contains a Hamiltonian double ray(Hamiltonian circle). 
\end{conj}
In~\Cref{equivalence} we proved the following.
Let $G=A\ast_K B$, where $A=K\sqcup Kb$ and $B=K\sqcup Kb'$ such that $b^2=b'^2$.
Then if $A$ and $B$ are generalized quasi-dihedral on $K$, so is $G$.

\begin{prob}
Can we drop the condition `` $b^2=b'^2$" in~\Cref{equivalence}?
\end{prob}

\subsection{Hamilton-connectivity}

A finite graph is called \emph{Hamilton-connected} if there is a Hamiltonian path between any pair of vertices of the graph. A finite bipartite graph is called \emph{Hamilton-laceable} if there is Hamiltonian path between pair of vertices at odd distance. Note that it is an easy observation that Hamilton-laceability (and not -connectivity) is the most that one can hope for in the case of finite bipartite graphs. The most celebrated theorem relating Cayley graphs and Hamilton-connectivity is due to Chen and Quimpo~\cite{chen1981strongly}, who proved something stronger about the Hamiltonicity of Cayley graphs of finite abelian groups with degree at least three; they are actually Hamilton-connected, unless they are bipartite when they are Hamilton-laceable.

Alspach et al.~prove in~\cite{alspach2010hamilton} extended this for Cayley graphs of finite generalized dihedral groups with degree at least three. 
Their proof also relies on extracting spanning finite grids(finite versions of infinite two-ended grids) or twisted cubic cylinders ---defined as \emph{honeycomb toroidal graphs} in ~\cite{alspach2010hamilton}---, where the double rays of the rows are replaced with finite cycles. A large part of their paper revolves around proving that these finite graphs are Hamilton-connected or -laceable.

Notice that from the definition of Hamiltonian connectivity we immediately obtain a Hamiltonian cycle when the two endvertices are connected with an edge. Using this, let us generalize the notion of Hamilton-connectivity to the infinite setting. We say an infinite (bipartite) graph is \emph{Hamiltonian ray-connected (-laceable)} if for any pair of vertices $u,v$ (at odd distance) there are two disjoint rays starting from $u,v$ spanning all the vertices of the graph. Similarly, an infinite (bipartite) graph is \emph{Hamiltonian arc-connected (-laceable)} if for any pair of vertices $u,v$ (at odd distance) there are two rays starting from $u,v$ and a double ray, all pairwise disjoint, whose union spans all the vertices of the graph. Observe that when $u$ and $v$ are connected with an edge in both definitions, we directly obtain a Hamiltonian double ray and a Hamiltonian circle, respectively.

The fact that bipartite graphs can only be Hamilton-laceable fails for infinite graphs. It is very easy (but a bit tedious as well) to prove that two-ended infinite grids are not just Hamilton-laceable, but Hamilton-connected.
\begin{prob}\label{Hamiltonian connectivity grid}
Is $\cay(G;S)$ both Hamiltonian ray- and arc-connected, where $G$ is a generalized quasi-dihedral group?
\end{prob}

\subsection{Decomposition and Uniqueness}
Alspach, in  {\cite[Unsolved Problem~4.5, p.~454]{alspach1985}}] conjectured the following.
Let $G$ be an Abelian group with a symmetric generating set~$S$.
If the Cayley graph\/ $\cay(G, S)$ has no loops, then it can be decomposed into Hamiltonian cycles \textup(plus a $1$-factor if the valency is odd\textup) and see \cite{arc0disjoint}, for a directed version of this problem,
Erde and Lehner \cite{ErdeLehner} showed every 4-regular Cayley graph of an infinite abelian group all of
whose finite cuts are even can be decomposed into Hamiltonian double rays, and so characterized when such decompositions exist.
This raises the following question about two-ended generalized quasi-dihedral groups.

\begin{prob}
Can every $4$-regular Cayley graph of a two-ended generalized quasi-dihedral group be decomposed into two hamiltonian double-rays(Hamiltonian circle)?
\end{prob}

Finally, consider $D_{\infty}=\langle a,b\mid b^2=(ba)^2=1\rangle $.
The Cayley graph of $D_{\infty}$ with respect to $S=\{ba,b,ab\}$ contains a unique Hamiltonian circle.
Naturally, we ask the following question.

\begin{prob}
Let $G$ be a generalized quasi-dihedral group.
For which generating $S$ of $G$ does $\cay(G;S)$ contain a unique Hamiltonian circle? 
\end{prob}

The above question has been answered \cite{MR4828039}


\end{section}

\noindent {\bf{Data availability}}\\
\noindent There are no associated data with this manuscript. This is a mathematical paper. No data were gathered, analysed, or used in any manner.
	\bibliographystyle{plain}
	\bibliography{collective.bib} 
\end{document}